\documentclass[11pt,twoside]{amsart}
\usepackage{amsmath,amssymb,amsthm}
\usepackage{dsfont}
\usepackage{xcolor}
\newtheorem{thm}{Theorem}[section]

 \newtheorem{lem}{Lemma}[section]
 \newtheorem{prop}{Proposition}[section]
 
\newtheorem{rem}{Remark}[section]

\newcommand{\bFormula}[1]{\begin{equation} \label{#1}}
\newcommand{\eF}{\end{equation}}

\newcommand{\Div}{{\rm div}_x}
\newcommand{\Grad}{\nabla_x}
\newcommand{\vr}{\varrho}
\newcommand{\tildrho}{\tilde{\varrho}}
\newcommand{\bTheorem}[1]{\begin{Theorem} \label{T#1} }
\newcommand{\eT}{\end{Theorem}}
\newcommand{\bProposition}[1]{\begin{Proposition} \label{P#1}}
\newcommand{\eP}{\end{Proposition}}
\newcommand{\bLemma}[1]{\begin{Lemma} \label{L#1} }
\newcommand{\eL}{\end{Lemma}}
\newcommand{\bCorollary}[1]{\begin{Corollary} \label{C#1} }
\newcommand{\eC}{\end{Corollary}}

\newcommand{\A}{\mathcal{A}}

\newcommand\R{\mathbb{R}}

\newcommand{\ds}{\displaystyle}

\newcommand{\Tr}{\hbox{\rm{Tr}\,}}

\renewcommand{\div}{\mbox{\rm div}\;\!}

\def\cC{{\mathcal C}}

\def\cF{{\mathcal F}}

\newcommand{\andf}{\quad\hbox{and}\quad}
\renewcommand{\Re}{\textnormal{Re}}

\usepackage[top=3cm,bottom=3cm,left=3cm,right=3cm]{geometry}

\usepackage{hyperref,enumitem}
\usepackage{cleveref}
\hypersetup{colorlinks=true,linkcolor=red,citecolor=blue,filecolor=magenta,urlcolor=cyan} 

\usepackage[sort,nocompress]{cite}

\begin{document}

\title[On an Euler-Schr\"{o}dinger system appearing in laser-plasma interaction]{On an Euler-Schr\"{o}dinger system appearing in laser-plasma interaction}
\author{K. Bhandari \and B. Ducomet \and \v S. Ne\v{c}asov\'{a} \and J. S. H. Simon}

\begin{abstract}
 We consider the Cauchy problem for the barotropic Euler system coupled to a vector Schr\"{o}dinger equation in the whole space.
Assuming that the initial density and vector potential are small enough, and that the initial velocity is  close to some
reference vector field $u_0$ such that the spectrum of $Du_0$ is bounded away from zero, we prove the existence of a global-in-time unique
solution with (fractional) Sobolev regularity. Moreover, we  obtain some algebraic time decay estimates of the solution.
Our work extends the papers  by D. Serre and M. Grassin  \cite{GS,G2,S} and previous works  by B. Ducomet and co-authors \cite{BDDN,DD}
dedicated to the compressible Euler-Poisson system. 
\end{abstract}

\maketitle

\vspace{.1cm}


\section{Introduction}

 We consider a simple hydrodynamic
 model appearing in the context of laser-plasma interaction, introduced by R. Sentis in \cite{S1,S2}. The fluid system under study is the compressible Euler system  coupled to a vector Schr\"{o}dinger equation. 
 More precisely, the system reads as:
\bFormula{i1g}
\partial_t \vr + \Div (\vr u) = 0,
\eF
\bFormula{i2g}
\partial_t (\vr u) + \Div (\vr u \otimes u) + \Grad\Pi(\vr)
=-\gamma_p\vr\Grad (A^2),
\eF
\bFormula{i3g}
\frac{2i}{c}\partial_t A+\frac{1}{k_0}\Delta_x A
+\zeta_1 A
-\zeta_2\vr A=0,
\eF
with initial data
\bFormula{i0g}
(\vr,u,A)(x,0)=(\vr_0,u_0,A_0)(x),
\eF
where $\vr=\vr(t,x)$ is the fluid density, $u=u(t,x)$ is  the fluid velocity, and $A=A(t,x)$ is the (complex) vector potential representing the laser field.  Moreover,   
$\Pi(\vr)=K\vr^\gamma$ is the barotropic pressure  with $K>0$ and the adiabatic exponent $\gamma>1$, and $A^2 = \sum_{j=1}^d A_j \overline{A_j}$. Finally, in the equation \eqref{i3g},  $\zeta_1=k_0+i\eta_0$ and
$\zeta_2=\frac{1}{\varrho_c}\left(k_0-i\nu_0\right)$ are two complex coefficients. 

The various physical constants are as follows: 
\begin{itemize}  
\item[--] $\gamma_p : = \frac{Zq_e^2}{4m_i m_e \omega_0^2}$, 
where $Z>0$ is the ionization level,
 $q_e$ is the electronic charge,
 $m_i$ is the ionic mass, $m_e$ is the electronic mass and 
$\omega_0$ is the laser pulsation; 

\item[--] 
$k_0$ is  the wave number of the laser in the vacuum;

\item[--]
$c$ is the velocity of light;

\item[--] 
$\varrho_c =\frac{\omega_0^2\varepsilon_0 m_e}{q_e^2}$, where $\varepsilon_0 >0$ is the vacuum permeability;  

\item[--] $\eta_0>0$ and  $\nu_0=\frac{\nu_{ei}}{c}$, where $\nu_{ei}$ is the electron-ion collision frequency.
\end{itemize}

\smallbreak
 The above system is a simplified model for the 
evolution of a plasma interacting with an electromagnetic field, after an enveloppe approximation introduced by R. Sentis in \cite{S1} (see also \cite{S2} for a general context).

As it will be shown below, the only static solution of the system is the trivial one $(\vr=0,u=0,A=0)$, and it is thus natural to focus on  initial densities that are either compactly supported or tend to $0$ at infinity. 
 In order to solve the Cauchy problem for \eqref{i1g}--\eqref{i3g},
 the standard approach is
 to use an energy  method after ``symmetrizing'' the system.
 In our context where the density may tend to $0$ at infinity,
Makino's  symmetrization  (first introduced in \cite{Ma}) is the most appropriate one. 

Problem \eqref{i1g}--\eqref{i3g} may be solved locally in time in the Sobolev space setting
by  means of a suitable  adaptation of the standard theory for first order  
quasilinear symmetric hyperbolic systems (see \cite{be}).  
%
Our first aim  is to  prove well-posedness for  system \eqref{i1g}--\eqref{i3g} in a functional setting that encompasses Sobolev spaces. In this context, one can  mention the works  by Chemin \cite{C} and  Gamblin \cite{ga}, 
based on the use of uniformly local Sobolev spaces (that
does not allow to keep track of the behavior of the solution at infinity), 
and the recent work by  Brauer and Karp \cite{BK} in weighted
Sobolev spaces. 
Compared to that latter  work, our approach is particularly elementary, as it
just requires basic Lebesgue and Sobolev spaces. 
 
 Furthermore, our functional framework turns out to be also relevant for 
 reaching our second aim, which is  to produce  global-in-time strong solutions for \eqref{i1g}--\eqref{i3g}. Clearly,  as no dissipative process takes place in the system,
generic data  (even small) are expected to lead to  only local-in-time 
solutions and examples of finite time blow-up are known  (see \cite{C,CW,R}).
  However, in a series of papers \cite{GS,G2,S}, Grassin and Serre  pointed out 
 that, under a  ``dispersive'' spectral  condition  on the initial velocity,
 and a smallness assumption on the initial density, the compressible Euler system admits a unique global smooth solution.  Here,  after our recent works in \cite{BDDN} and \cite{DD} 
 we want to adapt  Grassin and Serre's result  to the Euler-Schr\"odinger system 
 within a functional framework that includes the trivial solution $(\varrho=0, u=0, A=0)$
 and allows to show  its stability.

\medbreak

{\bf Paper organization.} The rest of the paper is structured as follows. 
In the next section, we state our main results and give some insights on our strategy. 
In Section \ref{aux}, we establish decay estimates in Sobolev spaces for our  compressible Euler-Schr\"{o}dinger system. 
Section \ref{Proof-main-result} is devoted to the proofs of  the global existence 
results, and then we show the uniqueness of the solution.
Some technical results like, in particular, first and second order commutator 
estimates are proved in Appendix \ref{appendix}. 

\medbreak

{\bf Notations.} Throughout the paper, $C$ denotes a generic  constant that may vary from line.
Sometimes, we use  the notation $\Phi\lesssim \Psi$ to mean that 
$\Phi\leq C\Psi.$  The notation $\Phi\approx \Psi$  is used if  both properties $\Phi\lesssim \Psi$ and $\Psi\lesssim \Phi$ hold true. 

By ${\rm Sp}\,(\A)$, we  denote the spectrum of any matrix or operator $\A.$  Finally, we shall denote by $\dot H^s$ and $H^s$ the homogeneous and nonhomogeneous Sobolev
spaces of order $s$ on $\mathbb R^d$.


\section{Main results}\label{main}

Since our functional framework will force the density to tend to $0$ at infinity, the standard  
symmetrization for the compressible Euler equations is no longer appropriate, we thus use the one that has been introduced by T. Makino in \cite{Ma}; more precisely, we set 
\begin{equation}\label{eq:mak}
\tilde \varrho:=\frac{2\sqrt{K\gamma}}{\gamma-1}\,\varrho^{\frac{\gamma-1}2}, \ \ \gamma> 1 .
\end{equation}
After this change of variable, the system \eqref{i1g}--\eqref{i3g} rewrites
\begin{align}\tag{ES}
\begin{cases} 
 \ds (\partial_t+u\cdot\nabla_x)\tilde \vr+\frac{\gamma-1}2 \tilde \vr\,\Div u=0, \\
 \ds (\partial_t+u\cdot\nabla_x)u+\frac{\gamma-1}2\tilde \vr \,\nabla_x\tilde \vr=-\gamma_p\nabla_x A^2, \\
\ds  \frac{2i}{c}\partial_t A+\frac{1}{k_0}\Delta_x A
+\zeta A
-\zeta' \tilde \vr^{\frac{2}{\gamma-1}} A=0\cdotp
\end{cases}\label{eq:MP}
\end{align} 
where $\zeta=\zeta_1$ and $\zeta'=\left(\frac{\gamma-1}{2\sqrt{K\gamma}}\right)^{\frac{2}{\gamma-1}}\zeta_2.$

\medbreak
Let us consider an auxiliary Cauchy problem for the $d$-dimensional Burgers equation:
\begin{equation}
\label{auxiso}
\partial_t  v +  v \cdot \Grad  v = 0,
\end{equation}
with initial data 
\bFormula{auxiso0}
 v(0,x)=v_0(x).
\eF
In the particular case of the 
Euler equation (that is $\gamma_p=0$) and under suitable spectral conditions on $Du_0,$
it is known from \cite{G,GS} that \eqref{auxiso} is a good approximation  of  \eqref{i1g}--\eqref{i2g}
provided the density of the fluid is small enough.
The formal heuristics is that if one neglects $\tilde \vr$  in the velocity
equation of \eqref{eq:MP} (as well as if $\gamma_p=0$), then one readily gets
\eqref{auxiso}.

\medbreak

In order to state our results, we need to introduce the following function  space: 
$$E^s:=\bigl\{z\in\cC(\R^d;\R^d); \; Dz\in L^\infty \text{ and } D^2z\in H^{s-2}\bigr\} . $$
The following regularity result  has  been first proved in \cite{G,GS} in the case of
{\em integer} exponents and extended in \cite{BDDN} for the  \emph{real} exponents. 

\begin{prop}
\label{p:Burgers}
Let $v_0$ be in $E^s$ with  $s>1+d/2$ and satisfies the following: 
\begin{enumerate}[label=$\bf (H0)$,leftmargin=1.5cm]
\item\label{spectral_gap} there exists  $\varepsilon>0$  such that for any   $x\in\R^d$, ${\rm dist}({\rm Sp}\,(Dv_0(x)),\R_-)\geq \varepsilon$. 
 \end{enumerate}
Then 
\eqref{auxiso}--\eqref{auxiso0} has a classical solution $v$ on $\R_+\times\R^d$ such that
\[
  D^2v \in \cC^j\bigl(\R_+;H^{s-2-j}(\R^d)\bigr)\ \ \ \mbox{for}\ \ j=0,1.
\]
Moreover,  $D  v\in \cC_b(\R_+\times\R^d)$ and we have for any $t\geq 0$ and  $x\in \R^d,$
\begin{equation}
\label{Du}
D  v(t,x)= \frac{I_d}{(1+t)} + \frac{F(t,x)}{(1+t)^2} 
\end{equation}
for some function $F\in\cC_b(\R_+\times\R^d;\R^d\times\R^d)$ that satisfies
\begin{equation}
\|F(t)\|_{\dot H^{\sigma}}\leq K_\sigma(1+t)^{\frac d2-\sigma}\ \ \mbox{for all }\ 0<\sigma\leq s-1,
\label{estim2}
\end{equation}
where $K_\sigma>0$ is constant. 

Finally, if $D^2v_0$ is bounded,  we have
\begin{equation}
{\displaystyle \|D^2  v(t)\|_{L^\infty}\leq \frac{C}{(1+t)^{3}}}.
\label{estim3}
\end{equation}
\end{prop}

\medbreak 

Let us write the  main results of the present paper. 
\begin{thm}[\textbf{Existence of global solution}]
\label{isentro}
Let $s>1+d/2$ and $\gamma>1.$  Assume that the initial data $(\varrho_0,u_0,A_0)$ satisfy:
\vspace*{.1cm}
\begin{itemize}
[label=$\bf (H1)$,leftmargin=1.5cm]
\item\label{Hypo-1}
 there exists $v_0 \in E^{s+1}$ verifying  \ref{spectral_gap}
and that $u_0-v_0$ is small in $H^s$;
\end{itemize}
\begin{itemize}
[label=$\bf (H2)$,leftmargin=1.5cm]
\item\label{Hypo-2} $\vr_0^{\frac{\gamma-1}{2}}$ is small  in $H^s;$
\end{itemize}
\begin{itemize}
[label=$\bf (H3)$,leftmargin=1.5cm]
\item\label{Hypo-3} $A_0$ is small  in $H^s.$
\end{itemize}

Denote by  $ v$  the global solution of \eqref{auxiso}--\eqref{auxiso0}  as given in  Proposition
\ref{p:Burgers}. 
Then, there exists a unique global solution $\left(\vr,u,A\right)$ to
\eqref{i1g}--\eqref{i3g}, such that 
\[
\bigl(\vr^{\frac{\gamma-1}{2}},u- v,A\bigr) \in
\cC\bigl(\R_+;H^{s}\bigl(\R^d\bigr) \bigr),
\]
provided $d,$ $\gamma$ and $s$ satisfy the following additional condition:
$$
s>\max\{2d+3,s_+\},
$$
where $s_+$ is given by \eqref{eq:condPs}.
\end{thm}
\begin{rem}\label{r:finitemass} The fact that $\varrho^{\frac{\gamma-1}2}\in \cC(\R_+;H^s(\R^d))$
implies that $\varrho\in\cC(\R_+;L^1(\R^d))$ whenever $\gamma\leq2$. 
Hence, the constructed solutions have finite mass, and one can show that it is conserved
through the evolution.
\end{rem}

Furthermore, the following decay estimates  are satisfied by  the global solutions constructed in the previous theorem.

\begin{thm}[\textbf{Decay estimates}]\label{decay}
Let all the assumptions of  Theorem \ref{isentro} be in force.  Then, for all $\sigma$ in $[0,s],$
 the solution $(\vr,u,A)$ constructed therein satisfies
$$
\left\|\vr^{\frac{\gamma-1}{2}},u- v, A\right\|_{\dot H^\sigma}\leq C (1+t)^{\frac d2-\sigma
-\min\left(1, \, d(\frac{\gamma-1}2)\right)},
$$
where $C$ depends only on the initial data, on $d,$ $\gamma,$
 and on $\sigma.$ 
\end{thm}


\section{Decay estimates in Sobolev spaces}\label{aux}

The goal of the present section is to prove  
a-priori  decay estimates in Sobolev spaces for the discrepancy between  
the solutions of \eqref{eq:MP}  and \eqref{auxiso} (i.e., Theorem \ref{decay}) . These estimates  will play a fundamental role 
in the proof of our global existence result.  




\bigskip 

Let  $v$ be the solution of the Burgers equation given by Proposition \ref{p:Burgers}. 
Let us set   $w:=u-v$ where $(\tilde \varrho,u,A)$ stands for a sufficiently smooth solution of \eqref{eq:MP} on $[0,T]\times\R^d.$ 
Then $(\tilde\varrho,w,A)$ satisfies:
\begin{align}\tag{BB}
\begin{cases}
\ds (\partial_t+w\cdot\nabla_x)\tilde\varrho+\frac{\gamma-1}2\tilde \varrho\,\div_x w+v\cdot\nabla_x\tilde\varrho+\frac{\gamma-1}2\tilde\varrho\,\div_x v=0,\\ 
\ds (\partial_t+w\cdot\nabla_x)w+\frac{\gamma-1}2\tilde\varrho\,\nabla_x\tilde \varrho+v\cdot\nabla_x w+w\cdot\nabla_x v=-\gamma_p\nabla_x \phi,\\
\ds \frac{2i}{c}\partial_t A+\frac{1}{k_0}\Delta_x A
+\zeta A
-\zeta'\tilde\varrho^{\frac{2}{\gamma-1}} A=0,
\end{cases}\label{eq:BB}
\end{align} 
where $\phi:=A^2$.

Our aim is to prove decay estimates in $\dot H^\sigma$ for the solution of \eqref{eq:BB} for all $0\leq\sigma\leq s.$
Clearly, arguing by interpolation, it suffices to consider the extreme cases $\sigma=0$ and $\sigma=s.$

\medbreak

(i) \underline{\em Case I:  $\sigma=0$}.  Taking the $L^2$ 
scalar product of the three equations of \eqref{eq:BB} with $(\tilde\varrho,w,\overline A)$ gives first
\begin{multline}\label{eq:L20}
\frac12\frac d{dt}\|(\tilde\varrho,w)\|_{L^2}^2-\frac12\int_{\R^d}(\tilde\varrho^2+|w|^2)\div_x v\,dx
-\frac12\int_{\R^d}(\tilde\varrho^2+|w|^2)\div_x w\,dx \\
+\frac{\gamma-1}2\int_{\R^d}\tilde\varrho^2\div_x v\,dx
+\int_{\R^d}(w\cdot\nabla_x v)\cdot w\,dx +\frac{\gamma-1}4 \int_{\R^d} \tilde\varrho^2 \div_x w\, dx \\
= -\gamma_p\int_{\R^d}\nabla_x\phi\cdot w\,dx,
\end{multline}
and
\begin{align}\label{eq:L21}
\frac{1}{c}\frac{d}{dt}\|\phi\|_{L^1}
+\eta_0\|\phi\|_{L^1}
+\left(\frac{\gamma-1}{2\sqrt{K\gamma}}\right)^{\frac{2}{\gamma-1}}\frac{\nu_0}{\varrho_c}\int_{\R^d}\tilde\varrho^{\frac{2}{\gamma-1}} \phi \ dx=0.
\end{align} 

We first observe that due to Proposition \ref{p:Burgers}
\begin{align*}
\Div v = \frac{d}{(1+t)}  + \frac{\Tr(F(t,x))}{(1+t)^2} .
\end{align*}
Then, considering 
\begin{equation}\label{eq:cdg} c_{d,\gamma}:= \min\left\{\eta_0,{(\gamma-1)}\frac{d}{2}\right\}-\frac d2 ,
\end{equation}
{}and  denoting by $M$,  the $L^\infty$-bound of $F$,   we deduce from \eqref{eq:L20} that,
\begin{multline}\label{eq:L2}
\min\Big\{\frac12, \frac1c\Big\}
\frac{d}{dt}\|(\tildrho,w,A)\|_{L^2}^2 + \frac{c_{d,\gamma}}{1+t}\|(\tildrho,w,A)\|_{L^2}^2
+\frac{\tilde\nu_0}{\varrho_c}\|\tildrho^{\frac{2}{\gamma-1}} \phi\|_{L^1} 
\leq
\gamma_p\|\nabla_x\phi\cdot w\|_{L^1}
\\+\frac{M\max\left\{1+\frac d2, (\gamma-1)\frac d2\right\}}{(1+t)^2}\|(\tildrho,w,A)\|_{L^2}^2+\max\left\{\frac12,\frac{\gamma-1}4\right\}\|\div_x w\|_{L^\infty}\|(\tildrho,w,A)\|_{L^2}^2.
\end{multline}
Here and in the sequel, we denote $\displaystyle \tilde \nu_0 = \left(\frac{\gamma-1}{2\sqrt{K\gamma}}\right)^{\frac{2}{\gamma-1}} \nu_0$.

Note that, even if  $\gamma_p=0$ (i.e., standard compressible Euler equation), proving existence results for \eqref{eq:BB} 
requires a control on $\|\div_x w\|_{L^\infty}.$ Owing to the hyperbolicity of the system, 
it seems (at least in the multi-dimensional case) difficult to go beyond the energy framework, 
and it is thus natural to look for a-priori estimates in  Sobolev spaces $H^s.$
Now, owing to Sobolev embedding, the minimal requirement to get eventually a bound on 
 $\|\Div w\|_{L^\infty}$ is that $s>1+d/2.$ 

  \medbreak
  
(ii) \underline{\em Case II:  $\sigma=s$}. In order to prove Sobolev estimates, we introduce  the homogeneous fractional derivation operator 
 $\dot\Lambda^s$  defined by  $\cF(\dot\Lambda^sf)(\xi):=|\xi|^s\cF f(\xi)$ and observe that  
 $\tildrho_s:=\dot\Lambda^s\tildrho,$ $w_s:=\dot\Lambda^sw$ and $\phi_s:=\dot\Lambda^s\phi$ satisfy
(with the usual summation convention over repeated indices)
\begin{align}\tag{BB$_s$}  
\begin{cases} 
\ds (\partial_t  + w \cdot \nabla_x)\tildrho_s + \frac{\gamma-1}{2} \tildrho \, \Div w_s + v \cdot \nabla_x \tildrho_s - s \partial_j v^k \dot \Lambda^{-2}\partial^2_{jk}\tildrho_s
+ \frac{\gamma-1}2 \dot \Lambda^s(\tildrho \, \Div v) \\ 
\hspace{10cm} \ds =\dot R^1_s + \dot R^2_s + \dot R^3_s,\\
\ds (\partial_t + w \cdot\nabla_x)w_s+\frac{\gamma-1}2\tildrho\nabla_x\tildrho_s+v\cdot\nabla_x w_s -s\partial_j v^k \dot \Lambda^{-2}\partial^2_{jk} w_s + \dot \Lambda^s(w\cdot\nabla_x v) \\
\hspace{8.5cm} \ds =\dot R^4_s+\dot R^5_s+\dot R^6_s+\gamma_p\nabla_x\phi_s,\\
\ds \frac{2i}{c}\partial_t A_s+\frac{1}{k_0}\Delta_x A_s
+\zeta A_s-\zeta' \dot\Lambda^s\left(\tildrho^{\frac{2}{\gamma-1}}A\right)=0.
\end{cases}\label{eq:BB_s}
\end{align}
with 
$$\begin{array}{ll}
\dot R^1_s : =[w,\dot\Lambda^s]\nabla_x\tildrho,\quad & \dot R^4_s:=[w,\dot\Lambda^s]\nabla_x w, \\ [1ex]
\dot R^2_s:= \frac{\gamma-1}2[\tildrho,\dot\Lambda^s]\Div w,\quad&\dot R^5_s:=\frac{\gamma-1}2[\tildrho,\dot\Lambda^s]\nabla_x\tildrho,\\[1ex]
\dot R^3_s:=[v,\dot\Lambda^s]\nabla_x\tildrho-s\partial_j v^k\dot\Lambda^{-2}\partial^2_{jk}\tildrho_s,\quad&
\dot R^6_s:=[v,\dot\Lambda^s]\nabla_x w-s\partial_jv^k\dot\Lambda^{-2}\partial^2_{jk} w_s. 
\end{array}
$$
The definition of $\dot R^3_s$ and $\dot R^6_s$ is motivated
by the fact that, according to  the classical  theory of pseudo-differential operators, we  expect to have 
$$
[\dot \Lambda^s,v]\cdot\nabla_x z=\frac1i\bigl\{|\xi|^s,v(x)\bigr\}(D)\nabla_x z +\hbox{remainder}.
$$
Computing  the \emph{Poisson bracket}  in the right-hand side  yields  
 $$
 \frac1i\bigl\{|\xi|^s,v(x)\bigr\}(D) = - s \partial_j v \dot\Lambda^{s-2} \partial_j.
 $$
 Now, taking advantage of 
 Proposition \ref{p:Burgers}, we get 
$$-\partial_j v^k\dot\Lambda^{-2}\partial^2_{jk}z=\frac1{1+t}z-\frac{F_{kj}}{(1+t)^2}\dot\Lambda^{-2}\partial^2_{jk}z,
$$
and also
$$\begin{aligned}
\dot\Lambda^s(\tildrho\,\div_x v)&=\frac d{1+t}\tildrho_s+\frac1{(1+t)^2}\dot\Lambda^s(\tildrho\,{\rm Tr}\, F)\\
\andf \dot \Lambda^s(w\cdot\nabla_x v)&=\frac1{1+t}w_s+\frac1{(1+t)^2}\dot\Lambda^s(F\cdot w).\end{aligned}$$
Hence, taking the $L^2$ inner product of \eqref{eq:BB_s} with $(\tildrho_s,w_s,A_s),$  and denoting 
\begin{equation}\label{eq:cdgs}c_{d,\gamma,s}:= c_{d,\gamma}+s,\end{equation}
 we discover that 
\begin{multline}\label{eq:Hs}
\min\Big\{\frac12,\frac1c\Big\}\frac d{dt}\|(\tildrho_s,w_s,A_s)\|_{L^2}^2
+\frac{c_{d,\gamma,s}}{1+t}\|(\tildrho_s,w_s,A_s)\|_{L^2}^2\\
\leq
\gamma_p\|\nabla_x\phi_s\cdot w_s\|_{L^1}\!
+\frac{\tilde \nu_0}{\varrho_c}\|\dot\Lambda^s\left(\tildrho^{\frac{2}{\gamma-1}} A\right)\|_{L^2}\|A_s\|_{L^2}
+\!\frac{\|\div_x
  w\|_{L^\infty}}2\|(\tildrho_s,w_s,A_s)\|_{L^2}^2 \\
  +\frac{\gamma-1}2\|\nabla_x \tildrho\|_{L^\infty}\|\tildrho_s\|_{L^2}\|w_s\|_{L^2}
+\frac{sM}{(1+t)^2} \|(\tildrho_s,w_s,A_s)\|_{L^2}^2\\+
\biggl[\frac{1}{(1\!+\!t)^2}\left(\frac{\gamma-1}2\|\dot\Lambda^s(\tildrho{\rm Tr}\, F)\|_{L^2}+\|\dot\Lambda^s(F\cdot w)\|_{L^2}\right) 
+\sum_{j=1}^6\|\dot R_s^j\|_{L^2}\biggr]\|(\tildrho_s,w_s,A_s)\|_{L^2}.
\end{multline}
The  terms $\dot R_s^1,$ $\dot R_s^2,$ $\dot R_s^4$ and $\dot R_s^5$  may be treated according to the homogeneous  version  of {\em Kato--Ponce} commutator estimates (see Lemma \ref{l:com1} in  the Appendix). 
We get
$$\begin{aligned}
\|\dot R^1_s\|_{L^2}&\lesssim \|\nabla_x\tildrho\|_{L^\infty}\|\nabla_x w\|_{\dot H^{s-1}}+\|\nabla_x w\|_{L^\infty}\|\tildrho\|_{\dot H^{s}},\\
\|\dot R_s^2\|_{L^2}&\lesssim \|\div_x w\|_{L^\infty}\|\tildrho\|_{\dot H^{s}}+\|\nabla_x \tildrho\|_{L^\infty}\|\div_x w\|_{\dot H^{s-1}},\\
\|\dot R_s^4\|_{L^2}&\lesssim \|\nabla_x w\|_{L^\infty}\|w\|_{\dot H^{s}},\\
\|\dot R_s^5\|_{L^2}&\lesssim \|\nabla_x\tildrho\|_{L^\infty}\|\tildrho\|_{\dot H^{s}}.\end{aligned}$$
The (more involved) terms $\dot R_s^3$ and $\dot R_s^6$  may be handled thanks to  Lemma~\ref{l:com2}.
We get
$$\begin{aligned}\|\dot R_s^3\|_{L^2}
&\lesssim \|\nabla_x\tildrho\|_{L^\infty}\|v\|_{\dot H^{s}}+\|\nabla_x^2v\|_{L^\infty}\|\nabla_x\tildrho\|_{\dot H^{s-2}},\\
\|\dot R_s^6\|_{L^2}&\lesssim \|\nabla_x w\|_{L^\infty}\|v\|_{\dot H^{s}}+\|\nabla_x^2v\|_{L^\infty}\|\nabla_x w\|_{\dot H^{s-2}}.\end{aligned}$$
Finally, the standard Sobolev type estimate yields 
$$\begin{aligned}
\|\dot\Lambda^s(\tildrho \, {\rm Tr}\, F)\|_{L^2}
&\lesssim\|\tildrho\|_{L^\infty}\|{\rm Tr}\, F\|_{\dot H^{s}} + \|{\rm Tr}\, F\|_{L^\infty}\|\tildrho\|_{\dot H^s} + \|\nabla_x \tildrho \|_{L^\infty} \|{\rm Tr} \, F\|_{\dot H^{s-1}}
\\
\andf 
\|\dot\Lambda^s(F\cdot w)\|_{L^2}&
\lesssim\|w\|_{L^\infty}\|F\|_{\dot H^{s}} + 
\|F\|_{L^\infty} \|w\|_{\dot H^s} + \|\nabla_x w\|_{L^\infty} \|F\|_{\dot H^{s-1}} . 
\end{aligned}
$$

Plugging all the above estimates in \eqref{eq:Hs} and using Proposition \ref{p:Burgers},   we end up with 
\begin{multline}\label{eq:Hs2}
\min\Big\{\frac12,\frac1c \Big\}\frac d{dt}\|(\tildrho,w,A)\|_{\dot H^s}^2+\frac{c_{d,\gamma,s}}{1+t}\|(\tildrho,w,A)\|_{\dot H^s}^2 \\
\lesssim
\gamma_p\|\nabla_x\phi_s\cdot w_s\|_{L^1}\!
+\frac{\tilde \nu_0}{\varrho_c}\|\dot\Lambda^s(\tildrho^{\frac{2}{\gamma-1}} A\|_{L^2}\|A_s\|_{L^2}\\+\frac{1}{(1+t)^2}\|(\tildrho,w,A)\|_{\dot H^s}^2
+\|D^2v\|_{\dot H^{s-1}}\|(\tildrho,w)\|_{L^\infty}\|(\tildrho,w,A)\|_{\dot H^s}
\\+\|(\nabla_x\tildrho,\nabla_x w)\|_{L^\infty}\|(\tildrho,w,A)\|_{\dot H^s}^2+\|\nabla^2_xv\|_{L^\infty}\|(\tildrho,w,A)\|_{\dot H^{s-1}}
\|(\tildrho,w,A)\|_{\dot H^s}.\end{multline}


\vspace*{.2cm}

Let us introduce the  notation
$$\dot X_\sigma:=\|(\tildrho,w,A)\|_{\dot H^\sigma}\andf
X_\sigma:=\sqrt{\dot X_0^2+\dot X_\sigma^2}\approx \|(\tildrho,w,A)\|_{H^\sigma}
 \quad\hbox{for }\ \sigma\geq0.$$
Our aim is to bound the right-hand side of \eqref{eq:L2} and  \eqref{eq:Hs2} in terms of $\dot X_0$ and $\dot X_s$ only. 

Arguing by interpolation, we get 
\begin{eqnarray}\label{eq:interpo1}
\|(\tildrho,w,A)\|_{L^\infty}&\!\!\!\lesssim\!\!\!& \dot X_0^{1-\frac d{2s}}\dot X_s^{\frac d{2s}},\\\label{eq:interpo2}
\|(D\tildrho,Dw,DA)\|_{L^\infty}&\!\!\!\lesssim\!\!\!& \dot X_0^{1-\frac1s(\frac d2+1)}\dot X_s^{\frac 1s(\frac d2+1)},\\ \label{eq:interpo3}
\|(\tildrho,w,A)\|_{\dot H^{s-1}}&\!\!\!\lesssim\!\!\!& \dot X_0^{\frac1s}\dot X_s^{1-\frac1s}.
\end{eqnarray}
Let us bound the electromagnetic contributions.
We first have
$$
\gamma_p\|\nabla_x\phi\cdot w\|_{L^1}
\leq
\gamma_p\|A^2\div_x w\|_{L^1}
\leq
\gamma_p\|A\|_{L^2}^2\|\div_x w\|_{L^\infty}
$$
$$
\lesssim
\|(\tildrho,w,A)\|_{L^2}^2\|\div_x w\|_{L^\infty}.
$$
From this, we get
$$
\gamma_p\|\nabla_x\phi\cdot w\|_{L^1}
\lesssim 
\dot X_0^{3-\frac{1}{s}(\frac{d}{2}+1)}
\dot X_s^{\frac{1}{s}(\frac{d}{2}+1)}.
$$
In the same stroke 
$$
\gamma_p\|\nabla_x\phi_s\cdot w_s\|_{L^1}
\leq
\gamma_p\|\nabla_x\phi\|_{\dot H^s}\| w_s\|_{L^2}
\leq
\gamma_p\|A^2\|_{\dot H^{s-1}}\| w_s\|_{L^2}$$
$$
\lesssim
\|A\|_{L^\infty}\|A\|_{\dot H^{s-1}}\| w_s\|_{L^2},
$$
which gives
$$
\gamma_p\|\nabla_x\phi_s\cdot w_s\|_{L^1}
\lesssim
\dot X_0^{1-\frac{d}{2s}+\frac{1}{s}}
\dot X_s^{2+\frac{d}{2s}-\frac{1}{s}}.
$$

Finally
$$
\|\dot\Lambda^s(\tildrho^{\frac{2}{\gamma-1}} A)\|_{L^2}
\leq
\|\dot\Lambda^s(\tildrho^{\frac{2}{\gamma-1}} )\|_{L^2}\|A\|_{L^2}
+\|\tildrho^{\frac{2}{\gamma-1}} \|_{L^2}\|A_s\|_{L^2}
$$
$$
\leq
\|\tildrho^{\frac{2}{\gamma-1}} \|_{\dot H^s}\|A\|_{L^2}
+\|\tildrho^{\frac{2}{\gamma-1}} \|_{L^2}\|A_s\|_{L^2}.
$$
But
$$
\|\tildrho^{\frac{2}{\gamma-1}} \|_{\dot H^s}
\leq
\|\tildrho\|_{L^\infty}^{\frac{2}{\gamma-1}-1} \|\tildrho_s\|_{L^2},
$$
and
$$
\|\tildrho^{\frac{2}{\gamma-1}} \|_{L^2}
\leq
\|\tildrho\|_{L^q}^{\frac{2}{\gamma-1}},
$$
with $q=\frac{4}{\gamma-1}$, and by interpolation
$$
\|\tildrho\|_{L^q}^{\frac{2}{\gamma-1}}
\leq
\left( \|\tildrho\|_{L^2}^{1-\delta}\| \tildrho \|_{\dot H^s}^{\delta}
\right)^{\frac{2}{\gamma-1}},
$$
with $\delta=\frac{d}{s}\left(\frac{1}{2}-\frac{1}{q}\right)$.

Therefore,
$$
\frac{\tilde \nu_0}{\varrho_c}\|\dot\Lambda^s(\tildrho^{\frac{2}{\gamma-1}} A\|_{L^2}\|A_s\|_{L^2}
\lesssim
\|\tildrho\|_{L^\infty}^{\frac{2}{\gamma-1}-1} \|\tildrho_s\|_{L^2}
\|A\|_{L^2}\|A_s\|_{L^2}
+\left( \|\tildrho\|_{L^2}^{1-\delta}\| \tildrho \|_{\dot H^s}^{\delta}
\right)^{\frac{2}{\gamma-1}}
\|A_s\|_{L^2}^2.
$$
So we end up  with
$$
\frac{\tilde \nu_0}{\varrho_c}\|\dot\Lambda^s(\tildrho^{\frac{2}{\gamma-1}} A\|_{L^2}\|A_s\|_{L^2}
\lesssim
\dot X_0^{1+(1-\frac{d}{2s})(\frac{2}{\gamma-1}-1)}
\dot X_s^{2+\frac{d}{2s}(\frac{2}{\gamma-1}-1)}
+
\dot X_0^{(1-\delta)\frac{2}{\gamma-1}}
\dot X_s^{2+\frac{2\delta}{\gamma-1}}.
$$

Then, plugging these inequalities and those of Proposition \ref{p:Burgers} in \eqref{eq:L2} and \eqref{eq:Hs2} yields
$$
\frac d{dt}\dot X_0+\frac{c_{d,\gamma}}{1+t}\dot X_0\lesssim \frac{\dot X_0}{(1+t)^2}+\dot X_0^{2-\frac1s(\frac d2+1)}\dot X_s^{\frac1s(\frac d2+1)},
$$
and
$$
\frac d{dt}\dot X_s+\frac{c_{d,\gamma,s}}{1+t}\dot X_s\lesssim \frac{\dot X_s}{(1+t)^2}+\frac{\dot X_0^{1-\frac d{2s}}\dot X_s^{\frac d{2s}}}{(1+t)^{s+2-\frac d2}}
+\dot X_0^{1-\frac 1s(\frac d2+1)}\dot X_s^{1+\frac1s(\frac d2+1)}+\frac{\dot X_0^{\frac1s}\dot X_s^{1-\frac1s}}{(1+t)^3}
$$
$$
+\dot X_0^{1-\frac{d}{2s}+\frac{1}{s}}
\dot X_s^{2+\frac{d}{2s}-\frac{1}{s}}
+
\dot X_0^{1+(1-\frac{d}{2s})(\frac{2}{\gamma-1}-1)}
\dot X_s^{2+\frac{d}{2s}(\frac{2}{\gamma-1}-1)}
+
\dot X_0^{(1-\delta)\frac{2}{\gamma-1}}
\dot X_s^{2+\frac{2\delta}{\gamma-1}}.
$$

Since we expect to have  $\dot X_\sigma\lesssim (1+t)^{-c_{d,\gamma,\sigma}}$ for all $\sigma\in[0,s],$
it is natural  to introduce the function $\dot Y_\sigma:=(1+t)^{c_{d,\gamma,\sigma}}\dot X_\sigma.$
However, for technical reasons, we proceed as in \cite{G2} and  work with 
\begin{equation}\label{eq:Y}
\dot Y_\sigma:=(1+t)^{c_{d,\gamma,\sigma}-a}\dot X_\sigma\quad\hbox{for some }\ a>1.\end{equation}
Then, observing that
$$
\frac d{dt}\dot Y_\sigma +\frac a{1+t}\dot Y_\sigma
=(1+t)^{c_{d,\gamma,\sigma}-a}\biggl(\frac d{dt}\dot X_\sigma+\frac{c_{d,\gamma,\sigma}}{1+t}\dot X_\sigma\biggr),
$$
the above inequalities for $\dot X_0$ and $\dot X_s$ lead us to
$$
\frac d{dt}\dot Y_0+\frac{a}{1+t}\dot Y_0\lesssim \frac{\dot Y_0}{(1+t)^2}
+\frac{\dot Y_0^{2-\frac1s(\frac d2+1)}\dot Y_s^{\frac1s(\frac d2+1)}}{(1+t)^{1+\frac d2+c_{d,\gamma}-a}},
$$
and
$$
\frac d{dt}\dot Y_s+\frac{a}{1+t}\dot Y_s\lesssim \frac{\dot Y_s}{(1+t)^2}+\frac{\dot Y_0^{1-\frac d{2s}}\dot Y_s^{\frac d{2s}}}{(1+t)^{2}}
+\frac{\dot Y_0^{1-\frac 1s(\frac d2+1)}\dot Y_s^{1+\frac1s(\frac d2+1)}}{(1+t)^{1+\frac d2+c_{d,\gamma}-a}}
+\frac{\dot Y_0^{\frac1s}\dot Y_s^{1-\frac1s}}{(1+t)^2}.
$$

 Introducing the notation 
$Y_\sigma:=\sqrt{\dot Y_0^2+\dot Y_\sigma^2},$ for any $\sigma \in [0,s]$, we get
\begin{align}\label{eq:Y_s}
\frac d{dt}Y_s+\frac a{1+t}Y_s
&\lesssim\frac{Y_s}{(1+t)^2}+\frac{Y_s^2}{(1+t)^{1+\frac d2+c_{d,\gamma}-a}} \\
& +\frac{Y_s^3}{(1+t)^{s+\frac{d}{2}+2(c_{d,\gamma}-a)}}
+
\frac{Y_s^{2+\frac{2}{\gamma-1}}}{(1+t)^{s+\frac{d}{2}(\frac{2}{\gamma-1}-1)+(c_{d,\gamma}-a)(\frac{2}{\gamma-1}+3)}}
\notag \\
&
+
\frac{Y_s^{2+\frac{2\delta}{\gamma-1}}}
{
(1+t)^{(1+\frac{2\delta}{\gamma-1})s+(c_{d,\gamma}-a)(\frac{2}{\gamma-1}+1)}
}. \notag 
\end{align} 
At this stage, one can take $a=1+\frac d2+c_{d,\gamma}$, so that $a>1$ holds true as soon as $\gamma>1.$
Then, we eventually get the differential inequality 
\begin{align}\label{eq:Y_s_2}
\frac{d}{dt}Y_s+\frac{a}{1+t}Y_s
&\lesssim\frac{Y_s}{(1+t)^2}+Y_s^2+\frac{Y_s^3}{(1+t)^{s+\frac{d}{2}+2(c_{d,\gamma}-a)}}
\\
&+
\frac{Y_s^{2+\frac{2}{\gamma-1}}}{(1+t)^{s+\frac{d}{2}(\frac{2}{\gamma-1}-1)+(c_{d,\gamma}-a)(\frac{2}{\gamma-1}+3)}}
+
\frac{Y_s^{2+\frac{2\delta}{\gamma-1}}}{(1+t)^{s(1+\frac{2\delta}{\gamma-1})+(c_{d,\gamma}-a)(\frac{2}{\gamma-1}+1)}}, \notag
\end{align}
where $a-c_{d,\gamma}=1+\frac{d}{2}$.

Now,  one may 
apply  Lemma \ref{l:ODE}, setting
$m_1=2$,
$m'_1=1- s-\frac{d}{2}+2(a-c_{d,\gamma})$,
$m_2=1+\frac{2}{\gamma-1}$,
$m'_2=1-s-\frac{d}{2}(\frac{2}{\gamma-1}-1)+(a-c_{d,\gamma})(\frac{2}{\gamma-1}+3)$,
$m_3=1+\frac{2\delta}{\gamma-1}$ and
$m'_3=1-s(1+\frac{2\delta}{\gamma-1})+(a-c_{d,\gamma})(\frac{2}{\gamma-1}+1)$,
and eventually get, provided $\|(\tildrho_0,w_0,A_0)\|_{H^s}$ is small enough:
\begin{equation}\label{eq:poissonHs}
\sqrt{(1+t)^{2s}\|(\tildrho,w,A)\|_{\dot H^s}^2+\|(\tildrho,w,A)\|_{L^2}^2}\leq 2\frac{e^{\frac{Ct}{1+t}}}{(1+t)^{c_{d,\gamma}}} \|(\tildrho_0,w_0,A_0)\|_{H^s}.
\end{equation}
Let us emphasize that in order to apply Lemma \ref{l:ODE}, we need  $a>1$ and 
$m'_j<m_ja,$ for $j=1,2,3$. After an elementary computation, one observes that the first condition is clearly satisfied while the second set of conditions leads first to
$$
s>1+\frac{d}{2},
$$
then to
$$
s>2d+3,
$$
and finally to
\begin{equation}
\label{eq:condPs}
s>s_+:=\frac{1}{2}[-\beta+\sqrt{\alpha^2+4\beta}],
\end{equation}
where 
$\alpha=1+\frac{2}{\gamma-1}+d(1+\frac{(\gamma-1)(\gamma-2)}{2(\gamma-1)})$ and 
$\beta=\frac{2d}{\gamma-1}(\frac{1}{2}-\frac{\gamma-1}{4})$.

In fact, one can conclude that \eqref{eq:poissonHs} holds true whenever $\gamma>1$ and
\begin{equation}
\label{eq:condP}
s>\max\{2d+3,s_+\}.
\end{equation}
This concludes the proof for Theorem \ref{decay}.  



\section{Proof of  Theorem \ref{isentro}}\label{Proof-main-result}

Let us give the sketch of the proof for global existence
of  \eqref{i1g}--\eqref{i0g}, and 
then establish  uniqueness by means of a classical energy method.

\subsection{Existence} 

Here we are given that $(\varrho_0,u_0,A_0)$ satisfy the assumptions of Theorem \ref{isentro}. 
One can observe that  proving the existence of a global-in-time solution $(\varrho, u, A)$ for the  system  \eqref{i1g}--\eqref{i0g} is equivalent to prove the same for the system \eqref{eq:MP} (after the change of variable \eqref{eq:mak}). This further leads to  prove   the existence of global-in-time solution for  the system \eqref{eq:BB}.

\subsubsection*{\underline{Step 1}: Solving an approximate system}

Fix some cut-off function $\chi\in \cC^\infty_c(\R^d)$ supported in, say, the ball $B(0,4/3)$ 
with  $0\leq \chi \leq 1$ in $B(0,4/3)$ and  $\chi=1$ in $B(0,3/4).$ 
Set $v^n:=\chi(n^{-1}\cdot)\,v.$ 
Let $J_n$ be the Friedrichs' truncation operator 
defined by  $J_n z:=\cF^{-1}(\mathds{1}_{B(0,n)} \cF z).$ 

For all $n\geq1$, we consider the following regularization  of \eqref{eq:BB}:
\begin{align}\tag{BB$_n$}
\begin{cases} 
\ds \partial_t\tilde\varrho+J_n((v^n\!+\!J_nw)\cdot\nabla_x J_n\tilde\varrho)+\frac{\gamma-1}2J_n(J_n\tilde\varrho\,\div_x(v^n\!+\!J_nw))=0,\\
\ds \partial_tw+J_n((v^n\!+\!J_nw)\cdot\nabla_x J_nw)+\frac{\gamma-1}2J_n(J_n\tilde\varrho\,\nabla_x J_n\tilde\varrho)+J_n(J_nw\cdot\nabla_x v^n)\\
\hspace{6cm} \ds = -\gamma_p\nabla_x J_n\phi,\\
\ds \frac{2i}{c}\partial_t A+\frac{1}{k_0}\Delta_x A
+\zeta A
-\zeta' J_n \left(\tilde\varrho^{\frac{2}{\gamma-1}} A\right)=0,
\end{cases}\label{eq:BB_n} 
\end{align} 
 supplemented with initial data $(J_n\tilde\varrho_0,J_nw_0,J_nA_0).$ 
\medbreak
Note that $v^n$ is in $\cC(\R_+;H^{s+1}).$ Hence the above system may be seen as an ODE in $L^2(\R^d;\R\times\R^d).$ Applying the standard Cauchy-Lipschitz theorem thus 
ensures that there exists a unique maximal solution $(\tilde \varrho^n,w^n,A^n)\in\cC^1([0,T^n);L^2)$ to \eqref{eq:BB_n}. 

Now, from  $J_n^2=J_n,$ we deduce that  $(J_n\tilde\varrho^n,J_nw^n,J_nA^n)$ also satisfies \eqref{eq:BB_n}. Hence, uniqueness of the solution 
entails that $J_n\tilde\varrho^n=\tilde\varrho^n$, $J_nw^n=w^n$ and $J_nA^n=A^n$.
 In other words, $(\tilde\varrho^n,w^n, A^n)$ is spectrally localized in the ball $B(0,n),$
and one can thus assert that  $(\tilde\varrho^n,w^n,A^n )\in\cC^1([0,T^n);H^\sigma)$ for all $\sigma\in \R,$ 
and  actually satisfies: 
\begin{align}
\begin{cases} 
\ds \partial_t \tilde\varrho^n +J_n((v^n\!+\!w^n)\cdot\nabla_x\tilde\varrho^n)+\frac{\gamma-1}2J_n(\tilde\varrho^n\,\div_x(v^n\!+\!w^n))=0,\\
\ds \partial_t w^n + J_n((v^n\!+\!w^n)\cdot\nabla_x w^n)+\frac{\gamma-1}2J_n(\tilde\varrho^n\,\nabla_x\tilde\varrho^n)+J_n(w^n\cdot\nabla_x v^n)\\
\hspace{6cm}\ds =-\gamma J_n\nabla_x \phi,\\
\ds \frac{2i}{c}\partial_t A^n+\frac{1}{k_0}\Delta A^n
+\zeta A^n
-\zeta'J_n\left(\tilde\varrho^{\frac{2}{\gamma-1}} A\right) = 0,\\
(\tilde\varrho^n,w^n,A^n)|_{t=0}=(J_n\tilde\varrho_0,J_nw_0,J_nA_0)).
\end{cases}\label{eq:J-BB_n}
\end{align}

\subsubsection*{\underline{Step 2}: Uniform a-priori estimates in the solution space}
Since  $J_n$ is an orthogonal projector in any Sobolev space and  $(J_n\tilde\varrho^n,J_n w^n,J_nA^n)=(\tilde\varrho^n,w^n,A^n),$
 one can repeat verbatim (and rigorously) the computations of Section \ref{aux}. The only change is that 
 since $Dv^n=Dv +{\mathcal O}(n^{-1}),$  the final estimates therein
 only hold in the time interval $[0,\min(cn,T_n))$ for some $c>0,$
 which eventually implies that  $T_n\geq cn.$
 

The conclusion of this step is that for any fixed $T>0,$ the triplet 
 $(\tilde\varrho^n,w^n, A^n)$  for $n$ large enough is defined on $[0,T],$  
 belongs to $\cC^1([0,T];H^\sigma)$ for all $\sigma\in\R$ and is
bounded in $L^\infty(0,T; H^s).$


\subsubsection*{\underline{Step 3}: Convergence}    Let us fix some $T>0.$ 
Given the uniform bounds of the previous step, the weak$^*$ compactness theorem 
ensures that (up to a subsequence),  there exists some $(\tilde\varrho,w,A) \in L^\infty(0,T ; H^s)$ such that 
$$(\tilde\varrho^n,w^n,A^n)\rightharpoonup (\tilde\varrho,w,A) \ \hbox{ weakly}^* \ \hbox{ in }\  L^\infty(0,T;H^s).$$
Furthermore, computing $\partial_t\tilde\varrho^n$, $\partial_t w^n$ and $\partial_t A^n$ by means of \eqref{eq:J-BB_n} and using 
standard product laws in Sobolev spaces, one can prove that  for all $\theta\in \cC^\infty_c(\R^d),$ 
$(\theta\partial_t\tilde\varrho^n,\theta\partial_t w^n,\theta\partial_t A^n)$ is bounded in $L^\infty(0,T; H^{s-1}).$ 
Hence, from  the Aubin-Lions lemma, interpolation and Cantor diagonal process, we obtain that (still up to a subsequence), 
$$(\theta\tilde\varrho^n,\theta w^n,\theta A^n)\to (\theta\tilde\varrho,\theta w,\theta A)\quad\hbox{in}\quad L^\infty(0,T;H^{s'})
\quad\hbox{for all }\ \theta\in \cC^\infty_c(\R^d)\andf s'<s.$$
This allows to pass to the limit in $\eqref{eq:J-BB_n}$ and to conclude that 
$(\tilde\varrho,w,A)$ satisfies \eqref{eq:BB} on $[0,T]\times\R^d.$ Of course, since $T>0$ is arbitrary, $(\tilde\varrho,w,A)$
actually satisfies  \eqref{eq:BB} on $\R_+\times\R^d$ and belongs to $L^\infty_{loc}(\R_+ ; H^s).$

\subsubsection*{\underline{Step 4}: Time continuity}  
That $(\tilde\varrho,w,A)$ lies in $\cC(\R_+;H^s)$ may be achieved either by adapting the arguments of Kato in \cite{K}
or those of \cite[Chap. 4]{BCD}. 
Note that ref. \cite{K} allows in addition to prove the continuity of the flow map in the space $\cC(\R_+;H^s).$

\subsection{Uniqueness}

This part is devoted to the proof of the  uniqueness result of the solution in terms of the following proposition.

\begin{prop}\label{locunik}
Let  the assumptions of Theorem \ref{isentro} be in force and 
assume that $(\varrho_1,u_1,A_1)$ and $(\varrho_2,u_2,A_2)$ are two solutions of \eqref{eq:MP} on $[0,T]\times\R^d$
such that $(\varrho_i)^{\frac{\gamma-1}2} \in \cC([0,T];L^2\cap L^\infty),$
$\nabla_x\varrho_i\in L^\infty\left((0,T)\times\R^d\right),$ 
$(u_i-v_i)\in \cC([0,T);L^2)$ and $\nabla_x u_i\in L^1(0,T ; L^\infty)$ for $i=1,2.$

\smallskip 

Then, if $(\varrho_1,u_1,A_1)$ and $(\varrho_2,u_2,A_2)$ coincide at time $t=0$  (the initial time), the two solutions are the same on $[0,T]\times\R^d.$ 
\end{prop}

\noindent 
{\bf Proof.}  We define $w_i = u_i-v_i$ for $i=1,2$, and
$$(\vr_\delta,w_\delta, \phi_\delta, A_\delta):=(\varrho_2-\varrho_1,w_2-w_1,\phi_2-\phi_1,A_2-A_1),$$
where $\phi_{i}:=(A_{i})^2$ for $i=1,2$.

Since $(\varrho_1,u_1,A_1)$ and $(\varrho_2,u_2,A_2)$ coincide at
time $t=0,$ so do $v_1$ and $v_2$. Hence, uniqueness for Burgers equation implies $v_1=v_2:=v$. We thus have  
\begin{align}
\begin{cases} 
\ds (\partial_t+w_2\cdot\nabla_x)\vr_\delta+\frac{\gamma-1}2\varrho_2\,\div_xw_\delta+v\cdot\nabla_x\vr_\delta+\frac{\gamma-1}2\vr_\delta\,\div_x v \\
\hspace{6cm} \ds =
-w_\delta\cdot\nabla_x\varrho_1-\frac{\gamma-1}2\vr_\delta\,\div_x w_1,\\
\ds (\partial_t+w_2 \cdot\nabla_x) w_\delta+\frac{\gamma-1}2\varrho_2\,\nabla_x\vr_\delta + v\cdot\nabla_x w_\delta+w_\delta\cdot\nabla_x v \\
\hspace{6cm}\ds =-w_\delta\cdot\nabla_x w_1 
         -\frac{\gamma-1}2\vr_\delta \nabla_x\varrho_1+\gamma_p\nabla_x\phi_\delta,\\
\ds \frac{2i}{c}\partial_t  A_\delta+\frac{1}{k_0}\Delta_x  A_\delta+\zeta  A_\delta
-\zeta'\left((\varrho_2)^{\frac{2}{\gamma-1}} A_2-(\varrho_1)^{\frac{2}{\gamma-1}} A_1\right)=0 .
\end{cases}
\end{align} 

Taking the $L^2$ scalar product of the first and second equations with $\vr_\delta$, $w_\delta,$ respectively,
and arguing as for proving  \eqref{eq:L20}, we get 
\begin{align}  \label{eq:uniq-1} 
\frac12 \frac{d}{dt}
\|(\vr_\delta,w_\delta)\|_{L^2}^2=\frac12\int_{\R^d}\left(\vr_\delta^2 + |w_\delta|^2\right)\div_x(v+w_2)\,dx
+\frac{\gamma-1}2
\int_{\R^d}\vr_\delta\,w_\delta\cdot\nabla_x (\varrho_2-\varrho_1)\,dx   \notag \\
-\int_{\R^d} \vr_\delta w_\delta\cdot\nabla_x\varrho_1 \, dx - \frac{\gamma-1}2\int_{\R^d}(\vr_\delta)^2\div_x(v+w_1)\,dx \notag \\
- \int_{\R^d}\left(w_\delta\cdot\nabla_x(v+w_1)\right)\cdot w_\delta \,dx +\gamma_p \int_{\R^d}\nabla_x \phi_\delta\cdot \overline{w_\delta} \, dx.
\end{align}
In the same fashion, taking $L^2$ scalar product the third equation with $A_\delta$, we get
\begin{align} 
\frac{1}{c}\frac d{dt}
\|A_\delta\|_{L^2}^2+\eta_0\| A_\delta\|_{L^2}^2 + \Re \int_{\R^d} i \zeta'\left((\varrho_2)^{\frac{2}{\gamma-1}} A_2-(\varrho_1)^{\frac{2}{\gamma-1}} A_1\right)\cdot \overline{A_\delta}\ dx=0.
\end{align}

Now adding these inequalities we can find that, 
$$
\frac d{dt}\|(\vr_\delta, w_\delta,A_\delta)\|_{L^2}^2\leq C
\|\nabla_x v,\nabla_x w_1,\nabla_x w_2,\nabla_x\varrho_1,\nabla_x\varrho_2\|_{L^\infty}
\|(\vr_\delta,w_\delta)\|_{L^2}^2+\gamma_p\|\nabla_x \phi_\delta\|_{L^2}\|w_\delta\|_{L^2}.
$$
After time integration, this gives  for all $t\in[0,T]$ (since $\varrho_\delta |_{t=0}=0$ and $w_\delta |_{t=0}=0$): 
\begin{align}\label{gro1}
&\|(\vr_\delta,w_\delta,A_\delta\|_{L^2})(t)\|^2_{L^2} \\ 
& \leq C\int_0^t\|(\nabla_x v,\nabla_x w_1,\nabla_x w_2,\nabla_x\varrho_1,\nabla_x\varrho_2)(\tau)\|_{L^\infty}
\|(\vr_\delta, w_\delta)(\tau)\|^2_{L^2}\,d\tau \notag  \\ 
& \quad +\gamma_p\int_0^t \|\nabla_x\phi_\delta(\tau)\|_{L^2} \|w_\delta(\tau)\|_{L^2} \, d \tau.
\notag 
\end{align} 
Rewriting the Schr\"{o}dinger equation for $A_\delta$, we have
$$
\frac{2i}{c}\partial_t A_\delta+\frac{1}{k_0}\Delta_x A_\delta+\zeta A_\delta =
\zeta^\prime \left[\left(\varrho_2^{\frac{2}{\gamma-1}} - \varrho_1^{\frac{2}{\gamma-1}}\right)A_2
+\varrho_1^{\frac{2}{\gamma-1}}A_\delta\right] .
$$
Now using the
 identity $$(\varrho_2)^{\frac2{\gamma-1}}-(\varrho_1)^{\frac2{\gamma-1}}=\frac2{\gamma-1}\vr_\delta\int_0^1 (\varrho_1+\eta\vr_\delta)^{\frac2{\gamma-1}-1}\,d\eta, $$
 and using the fact that the operator $\nabla_x  S(t)$ (where $S(t)$ is the Schr\"{o}dinger semigroup) maps $L^2$ in $L^2$
we get 
$$\|\nabla_xA _\delta\|_{L^2}\leq C\left(\|\vr_\delta\|_{L^2}\biggl\|\int_0^1 (\varrho_1+\eta\vr_\delta)^{\frac2{\gamma-1}-1}\,d\eta\biggr\|_{L^\infty}+\|A_\delta\|_{L^2}\right).$$
Noting that $\gamma\leq 1+2d/(d+2),$ and $\varrho_i\in L^\infty(0,T;L^2\cap L^\infty)$ for $i=1,2,$ the above integral (the coefficient of $\|\varrho_\delta\|_{L^2}$) in the right-hand side is  bounded  in terms of $\varrho_1$ and $\varrho_2$. Thus, we have 
\begin{equation}\label{aux-gro2}
\|\nabla_x A_\delta\|_{L^2} \leq \mathcal K_{\varrho_1, \varrho_2} \|\varrho_\delta\|_{L^2} + C \|A_\delta\|_{L^2} ,
\end{equation}
where $\mathcal K_{\varrho_1,\varrho_2}$ may depend on time `$t$' but not in the space variable `$x$'. 

Since $\phi_\delta= (A_2)^2-(A_1)^2$ and using the fact $(A_i)^2=\sum_{j=1}^d A_{i,j} \overline{A_{i,j}}=A_i \cdot \overline{A_i}$, $i=1,2$ we find that
\begin{align*}
\nabla_x \phi_\delta &= \nabla_x (  A_2\cdot \overline{A_2} -A_1\cdot \overline{A_1}  ) \\
& = \nabla_x (A_\delta \cdot \overline{A_2}  +  A_1 \cdot \overline{A_\delta}) \\
& = (\nabla_x A_\delta ) \overline{A_2} + A_\delta (\nabla_x \overline{A_2}) + 
(\nabla_x A_1) \overline{A_\delta} + A_1 (\nabla_x \overline{A_\delta}) .
\end{align*}
Therefore, 
\begin{align} 
\|\nabla_x\phi_\delta\|_{L^2} 
&\leq  C_{A_1, A_2} \left(\|A_\delta\|_{L^2}  + \|\nabla_x A_\delta\|_{L^2}  \right)  \notag \\
& \leq \mathcal K_{\varrho_1,\varrho_2,A_1, A_2} \left( \|\varrho_\delta\|_{L^2} + \|A_\delta\|_{L^2}  \right),
\label{gro2}
\end{align}
where we have used the bound \eqref{aux-gro2}, and the quantity $K_{\varrho_1,\varrho_2,A_1, A_2}$ above may depend on  time `$t$' but not in the space variable `$x$'.

Using \eqref{gro2} in \eqref{gro1}, we get 
\begin{align*}
\|(\vr_\delta,w_\delta,A_\delta\|_{L^2})(t)\|^2_{L^2}\leq \int_0^t \mathcal K_{v, \varrho_1, \varrho_2, w_1, w_2, A_1, A_2, \gamma_p} (\tau)  \| (\varrho_\delta, w_\delta, A_\delta)(\tau)\|^2_{L^2} \, d \tau ,
\end{align*}
for some function $\mathcal K_{v, \varrho_1, \varrho_2, w_1, w_2, A_1, A_2, \gamma_p}$ depending on the quantities appearing in the subscript.
Then, applying Gr\"{o}nwall's lemma to the above inequality ensures the  uniqueness of the solution. 

This completes the proof of the proposition. \qed

\bigskip 

\subsection*{Acknowledgments.}
We thank Xavier Blanc and Rapha\"{e}l Danchin for fruitful discussions about the model. K. Bhandari and \v{S}. Ne{\v{c}}asov{\'{a}} received funding from   the  Praemium Academiae of \v{S}.  Ne{\v{c}}asov{\'{a}}. The Institute of Mathematics, CAS is supported by RVO:67985840.

\appendix 

\section{Some technical results}\label{appendix}

 In this part, we prove here some technical results that have been used in this paper.  Let us start with an ODE estimate.

\begin{lem}\label{l:ODE} Let $Y:\R_+\to\R_+$ satisfy the differential inequality 
$$\frac d{dt}Y+\frac a{1+t}Y\leq C\biggl(\frac Y{(1+t)^2}+Y^2+\sum_{j\leq J}(1+t)^{m'_j-1}Y^{m_j+1}\biggr)\quad\hbox{on }\ \R_+$$
for some $J>0$, $C>0,$ $a>1,$ $m_j>0$ and $m'_j<ma_j$, for $j\leq J$. 
Then, there exists $c=c(a,m_j,m_j',C)$ such that if $Y(0)\leq c,$ then we have
$$
Y(t)\leq 2e^{\frac{Ct}{1+t}}\frac{Y_0}{(1+t)^a}\quad\hbox{for all }\ t\geq0.
$$
\end{lem}
\begin{proof}
We set $Z(t):=(1+t)^ae^{-\frac{Ct}{1+t}}Y(t)$ and observe that  the above differential inequality recasts in 
$$\frac d{dt} Z\leq C(1+t)^{-a}e^{\frac{Ct}{1+t}}Z^2
+C\sum_{j\leq J}(1+t)^{m_j'-m_ja-1}e^{\frac{Cmt}{1+t}}Z^{m_j+1},$$
which implies that
\begin{equation}\label{eq:Z1}\frac d{dt} Z\leq Ce^C(1+t)^{-a}Z^2+C\sum_{j\leq J}(1+t)^{m_j'-m_ja-1}e^{Cm_j}Z^{m_j+1}.\end{equation}
The conclusion stems from a  bootstrap argument : let $Z_0:=Z(0)$ and  assume that 
\begin{equation}\label{eq:Z2}
Z(t)\leq 2Z_0\quad\hbox{on}\quad [0,T].\end{equation}
 Then \eqref{eq:Z1} implies that
$$\frac d{dt} Z\leq 4Ce^C(1+t)^{-a}Z_0^2+C\sum_{j\leq J}(1+t)^{m_j'-m_ja-1}e^{Cm_j}(2Z_0)^{m_j+1}.$$
Hence, integrating in time, we discover that on $[0,T],$ we have
$$Z(t)\leq Z_0+\frac{4Ce^C}{a-1}Z_0^2\bigl(1-(1+t)^{1-a}\bigr)+C\sum_{j\leq J}\frac{e^{Cm_j}(2Z_0)^{m_j+1}}{m_ja-m_j'}\bigl(1-(1+t)^{m'_j-m_ja}\bigr)\cdotp$$
Let us discard the obvious case $Z_0=0.$ Then, if  $Z_0$ is so small as to satisfy 
$$\frac{4Ce^C}{a-1}Z_0+C\sum_{j\leq J}\frac{2^{m_j+1}Ce^{Cm_j} Z_0^m}{m_ja-m_j'}\leq 1,$$
the above inequality ensures that  we actually have $Z(t)<2Z_0$ on $[0,T]$. Therefore, the supremum of $T>0$ satisfying \eqref{eq:Z2} has to be infinite.
\end{proof}

We also used in the text the following first order  commutator estimate which
is a straightforward adaptation 
  to the homogeneous framework of the second inequality of
\cite[Lemma A.2]{BDD}:
\begin{lem}\label{l:com1} If $s>0,$ then we have:
$$\|[v,\dot\Lambda^s]u\|_{L^2}\lesssim
\|v\|_{\dot H^s}\|u\|_{L^\infty}+\|\nabla v\|_{L^\infty}\|u\|_{\dot H^{s-1}}.$$
\end{lem}

 \medskip 

The following second order commutator inequality (proved in \cite[Lemma A.4]{BDDN}) played a key role in the proof of Sobolev estimates with non-integer exponent for the solution to \eqref{eq:BB}. 

\begin{lem}\label{l:com2} If $s>1,$ then we have:
$$\|[v,\dot\Lambda^s]u-s\nabla v\cdot\dot\Lambda^{s-2}\nabla u\|_{L^2}\lesssim
\|v\|_{\dot H^s}\|u\|_{L^\infty}+\|\nabla^2v\|_{L^\infty}\|u\|_{\dot H^{s-2}}.$$
\end{lem}

\newpage

\centerline{Kuntal Bhandari}
\centerline{Institute of Mathematics of the Czech Academy of Sciences}
\centerline{\v Zitna 25, 115 67 Praha 1, Czech Republic}
 \centerline{E-mail: bhandari@math.cas.cz}
 \vskip0.5cm
\centerline{Bernard Ducomet}
 \centerline{Universit\'e Paris-Est}
\centerline{LAMA (UMR 8050), UPEMLV, UPEC, CNRS}
\centerline{ 61 Avenue du G\'en\'eral de Gaulle, F-94010 Cr\'eteil, France}
 \centerline{E-mail: bernard.ducomet@u-pec.fr}
 \vskip0.5cm
\centerline{\v S\'arka Ne\v casov\'a}
 \centerline{Institute of Mathematics of the Czech Academy of Sciences}
\centerline{\v Zitna 25, 115 67 Praha 1, Czech Republic}
 \centerline{E-mail: matus@math.cas.cz}
  \vskip0.5cm
\centerline{John Sebastian H. Simon}
\centerline{Mathematisches Institut}
	\centerline{Universit{\"a}t Koblenz}
	\centerline{Universit{\"a}tsstrasse 1, 56070 Koblenz, Germany}
	\centerline{E-mail:  jhsimon1729@gmail.com; jhsimon@uni-koblenz.de}

\end{document}